\documentclass[a4paper,12pt]{amsart}
\pdfoutput=1
\usepackage[protrusion=true,expansion=true]{microtype}
\usepackage[T1]{fontenc} 
\usepackage{mathtools}
\usepackage[all]{xy}
\usepackage[english]{babel}
\usepackage{frcursive}          
\usepackage{tikz}
\usetikzlibrary{patterns}
\usepackage{eucal}
\usepackage{graphicx}
\usepackage{mathrsfs}
\usepackage{amssymb}
\usepackage{amsxtra}
\usepackage{enumerate}
\usepackage{stmaryrd}

\newtheorem{theorem}{Theorem}[subsection]

\newtheorem{proposition}[theorem]{Proposition}
\newtheorem{lemma}[theorem]{Lemma}

\newtheorem{corollary}[theorem]{Corollary}

\newtheorem{definition}[theorem]{Definition}
\newtheorem{example}[theorem]{Example}

\def \1{\mathbb {1}}
\def \SM{\mathbb {S}}%        corps des reels
\def \RM{\mathbb {R}}%        corps des reels
\def \NM{\mathbb{N}}%        entiers naturels
\def \CM{\mathbb{C}}%        nombres complexes

 \def \Hom {{\rm Hom}}
 
 \def \Hom {{\rm Hom}}

\def \p {{\rm exp\,}}
\def \Id {{\rm Id\,}}

\def \d{\partial}%derivee partielle
 
\def\a{\alpha}
\def\b{\beta}

\def\l{\lambda}

\def\p{\varphi}

\def\G{\Gamma}   
\def\D{\Delta}
\def \s{\sigma}

\def \to{\longrightarrow} 

\def \alg{\mathfrak{g}}

\def \< {{\langle }}
\def \> {{\rangle }}
\def \( {\left( }
\def \) {\right) }

\newcommand{\Bt}{{\mathcal B}}
\newcommand{\Ct}{{\mathcal C}}

\newcommand{\Lt}{{\mathcal L}}
\newcommand{\Mt}{{\mathcal M}}

\newcommand{\Ot}{{\mathcal O}}

\newcommand{\Xt}{{\mathcal X}}

\newcommand{\lra}{\longrightarrow}
\newcommand{\Sets}{{\bf  Sets}}
\newcommand{\Vect}{{\bf  Vect}}
\newcommand{\Ban}{{\bf  Ban}}
\newcommand{\Kol}{{\bf  Kol}}

\title[Banach space functors]{A category  of \\ Banach space functors}
\author{Mauricio GARAY \MakeLowercase{ and} Duco van Straten.}

\begin{document}

%\normalsize
%\authorfootnotes

\begin{abstract}
We introduce a sheaf theoretic viewpoint on functional analysis designed for infinite dimensional Lie group actions. We develop functional calculus for Banach valued functors  and, in particular, prove the existence of an exponential map for a certain class of operators that generalise first order partial differential operators. \end{abstract}
\maketitle
 
%%%%%%%%%%%%%%%%%%%%%%%%%%%%%%%%%%%%%%%%%%%%%%%%%%%%%%%%%%%%%%%%%%%%%%%%%%%%%%%%
%%%%%%%%%%%%%%%%%%%%%%%%%%%%%%%%%%%%%%%%%%%%%%%%%%%%%%%%%%%%%%%%%%%%%%%%%%%%%%%
\section{\large \bf Introduction}
\subsection{Normal forms and versal deformations}
Lie group actions play an important role in many branches of mathematics: classification problems, representation theory, constructions of moduli spaces and stacks etc. At the frontier between mathematics and physics, one encounters loop groups, gauge groups, diffeomorphism groups which all turn out to be infinite dimensional. There is no general theory for the study of group actions in the infinite dimensional context and, when hard analysis turns out to be efficient --for instance in KAM theory-- the arguments are in general highly technical. Our aim is to provide a categorical framework from  which one can easily deduces general results on group actions, such as normal forms and versal deformations.

To explain these two notions --which are at the origin of our investigations-- consider the familiar finite dimensional situation of a Lie group $G$ acting on a vector space $E$. The {\em normal form problem} consists in finding a ``good representative'' in each orbit. The Jordan normal form for the adjoint action of $G=Gl(n,\CM)$ on $E=M(n,\RM)$ provides such an example, although the notion of ``goodness'' is subjective and related to the type of applications one has in mind.

\begin{figure}[htb!]
 \includegraphics[width=0.4\linewidth]{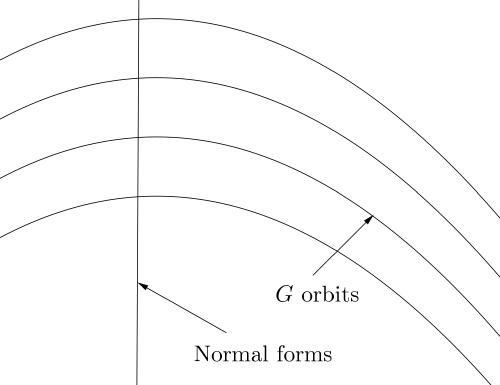}
\end{figure}
\ \\

Once a normal form is chosen, say $x$, the {\em versal deformation problem} around the point $x$ consists in finding an affine space $x+F \subset E$ which is transversal to the $G$-orbits. The family  $x+F$ is then called {\em a versal deformation of $x$}. Versal deformations might be thought as families of normal forms. Although for $G=GL(n,\RM)$ the subject is classical, versal deformations of minimal dimension for the adjoint action were given by Arnold in 1971 and extended by Brieskorn and Slodowy to simple Lie algebras~\cite{Arnold_matrices,Brieskorn_ICM,Slodowy}

\subsection{Infinitesimal action}
A basic technique for finding normal forms and versal deformations consists in looking at the infinitesimal action
$$\alg \times E \to E $$
of the Lie algebra $\alg$ on the tangent plane at the point $x$ which we identify with the vector space $E$. The question is then: does the surjectivity of the map
$$\rho:\alg \times F \to E,\ (v,t) \mapsto v(x)+t$$ implies that any element of a neighbourhood of $x$ can be taken back to an element of $F$?  

In finite dimensional spaces and even in Banach spaces, the answer is yes, as one can prove using the implicit function theorem. 
 
\subsection{The Lie iteration}
There is a more traditional way to prove that a subspace $F$ is transversal to the orbits than to use implicit function theory based on perturbation theory. We write an element in a neighbourhood of $x \in E$ as
$$x_0=x+r_0$$ 
Then we search $v_0 \in \alg $ and $t_0 \in F$ solving the {\em homological equation}:
$$v_0(x)=r_0+t_0 $$
and define a first element $g_0=e^{-v_0}$. We then have:
$$g_0(x+r_0)=x+t_0+r_1 .$$
We now repeat the procedure and try to find $v_1,t_1$ solving the equation:
$$v_1(x+t_0)=r_1+t_1 $$
and take  $g_1=e^{-v_1}$. We get now:
$$g_1g_0(x+t_0+t_1)=x+t_0+t_1+r_2 $$
and so on. This process --that we call the {\em Lie iteration}-- can be used in many places: differential equations, singularities of mappings, hypersurface singularities etc. The important point is that in many applications the homological equation can easily be solved along the transversal $x+F$.

\vskip0.3cm

\begin{figure}[htb!]
 \includegraphics[width=0.5\linewidth]{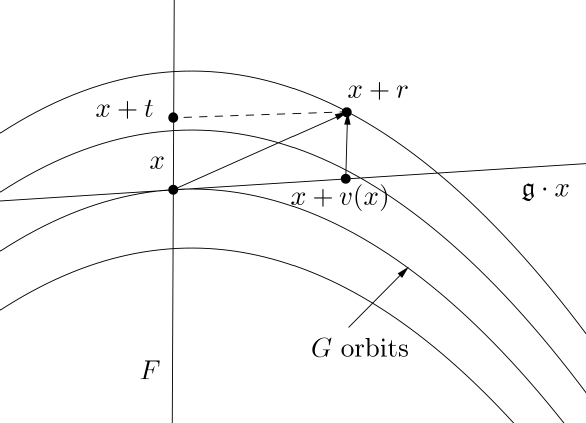}
\end{figure}
\ \\
%%%%%%%%%%%%%%%%%%%
\subsection{Convergence of the iteration}
The choice of a bounded operator 
$$u:E \to E$$ of a Banach algebra transforms the Banach space $E$ into a module over the ring of convergent power series having a convergence radius $R > \| u \|.$ This leads to the standard {\em functional calculus} which in particular establishes the existence of an exponential together with the estimate 
$$\| e^u \| \leq e^{\| u \|} .$$
If we now assume that the solutions $(v_n)$ to the homological equations and the elements of the transversal $(t_n)$ might be chosen so to define  summable norms sequences then the Lie iteration converges and
$$\|\prod_{i \geq 0} e^{v_i} \| \leq \prod_{i \geq 0} e^{\| v_i \|}$$
This is the only case where the convergence of  the Lie iteration is elementary. We refer to~\cite{KAM_theory_I} for details. In locally convex spaces, the situation is much more involved, as usually there is no functional calculus available.

%%%%%%%%%%%%%%%%%%%%%%%%%%%%%%%%%%%%%%%%%%%%%%%%%%%%%%%%%%%%%%%%%%%%%%%%%%%
\subsection{The exponential map}
In the sixties, infinite dimensional Lie groups appeared in gauge theories and in hydrodynamics. After Arnold proposed to include hydrodynamics as the geometry of the group of volume preserving diffeomorphisms, precise definitions were used to prove existence and unicity results~\cite{Arnold1966geometrie,Arnold_Khesin,Ebin_Marsden,leslie1967differential,Marsden1970hamiltonian,Omori1973groups}. Since then the study of infinite dimensional Lie group theory remained a subject of active research. We refer the interested reader to~\cite{Glokner2006,Khesin2009geometry,neeb2006towards,Omori2017infinite}.

Now, in the finite dimensional situation, a Lie algebra 
$\alg$ is related to its group $G$ via the exponential map
$$\alg \to G,\ v \mapsto e^v .$$
In the infinite dimensional situation, the situation is much more complicated. Even if we assume the exponential to be well-defined, it is in general not a local diffeomorphism. Milnor had already observed this fact for the group $Diff(\SM^1)$ ; his argument is quite simple: if a diffeomorphism turns out to be the exponential of a vector field, then linearising the vector field  would linearise the mapping as well, but any neighbourhood of the identity contains many non-linearisable diffeomorphisms~\cite[Warning 1.6]{Milnor1984remarks}. For locally convex spaces, no general theorem  stating the existence of such a mapping is known~(see \cite{neeb2006towards} for a survey).

In particular, the manifold structure of these infinite dimensional Lie groups is not given by the exponential but rather by special parametrisations. 
For instance, if $v$ is the germ of a vector field on $\RM^n$, then  
$\Id+v$ is the germ of a diffeomorphism and the map
$$v \mapsto \Id+v $$
can be used to  define an infinite dimensional Lie group structure for germs of diffeomorphism. 
 
 The problem is then that for these special parametrisations, simple operations rapidly become complicated~; for instance the inverse map of $\Id +v$ is given by Lagrange's inversion formula which is already non-trivial. This contrasts with the inverse of $e^v$ which is simply $e^{-v}$. Considerations of special  subgroups like symplectomorphisms or volume-preserving diffeomorphisms require even more complicated arguments.

%%%%%%%%%%%%%%%%%%%%%%%%%%%%%%%%%%%%%%%%%%%%%%%%%%%%%%%%%%%%%%%%%%%%%%%%%
\subsection{The Morse lemma}
What happens to the notion of normal forms and versal deformation in the infinite dimensional context? To get an idea of the complexity involved in the convergence of the Lie iteration, let us consider the action of germs of biholomorphisms on the algebra $\CM\{ x \}$ of convergent power series. To be even more precise, let us analyse the following elementary form of the Morse lemma:
\begin{lemma}
Assume $f=x^2+o(x^2) \in \CM\{ x \}$ then there exists a biholomorphic map-germ $x=\p(y)$ such that
$$f \circ \p(y)=y^2 $$
\end{lemma}
\begin{proof}
 The proof is straightforward. We write
 $$f(x)=x^2g(x),\ g(x)=1+o(1) $$
 and define $y=x\sqrt{g}$
\end{proof}
The locally convex Lie group of biholomorphic map germs
$$\p:(\CM,0) \to (\CM,0) $$
acts on the space $\CM\{ x \}$. The lemma says that the affine subspace $x^2+\Mt^3$ lies inside the orbit of $x^2$ where
$$\Mt=\{ f \in \CM\{ x \}:f(0)=0 \} $$
stands for the maximal ideal.

Let us now look at the solution given by the Lie iteration.
If we choose any vector field $v=a(x)\d_x$ where $a$ lies in $\Mt^2$ (the square of the maximal ideal of the local ring $\CM\{ x \}$), it is not difficult  to show the convergence of the formal power series
$$e^v=\sum_{n \geq 0} \frac{v^n}{n!} $$
So let us  write the Lie iteration starting from a particular case, say
$$f(x)=x^2+x^3 $$
Here $E=\Mt^2$ and $F=\{ 0 \}$ and the Lie algebra is given by germs of vector fields $a(x)\d_x,\ a \in \Mt^2$. We have 
$$v_0=\frac{x^2}{2}\d_x.$$ Then we need to compute
$$e^{-v_0}f=x^2+o(x^3) $$
and so on. Here the exponential might be interpreted either as the formal power series or as the time $1$ flow of the vector field but how to prove the convergence of the Lie iteration in this case? 

We would need estimates for the normalising fields $v_0,v_1,\dots$ and therefore introduce Banach norms on the locally convex space $\CM\{ x \}$ which induce operator norms on derivations. To construct these norms, one usually chooses open neighbourhoods in $\CM$ over which the functions are holomorphic and then try to control
the  effect of taking the exponential and how it transforms these domains.  Although this particular case of the Morse lemma is one of the simplest examples, the convergence of the iteration is already a non trivial exercise.

%%%%%%%%%%%%%%%%%%%%%%%%%%%%%%%%%%%%%%%%%%%%%%%%%%
\subsection{Lie iteration vs implicit function theorems}
So in applications, the convergence of the Lie  iteration or similar methods is hard to establish and is usually done by a long series of estimates, which require great skill. One of the first examples that displayed such virtuosity was given by Arnold in its original's proof of the KAM theorem~\cite{Arnold_KAM}. 

 The work of Moser represents a radical change. Moser proposed to base infinite dimensional group action on implicit function theorems. This project has been achieved with some success: Sergeraert proved versal deformation of mappings, Moser did find a simple proof to Arnold's theorem for vector fields which is the toy model of KAM theory~\cite{Bost_KAM,Fejoz_KAM,Moser_Pisa_2,Sergeraert,Zehnder_implicit}. 

So the use of implicit function theorems was regarded as the functional analytic solution to the long and tedious computations, as Zehnder enthusiastically wrote in 1972: {\em For
instance, while all previous proofs of the theorems in question involve infinitely many coordinate change of variables and consequently complicated convergence arguments, we shall avoid this inconvenience and work in a function space over a fixed set of variables~\cite{Zehnder_implicit}}. 

It took nevertheless more than 10 years before a baby model of the KAM theorem (Kolmogorov invariant torus theorem) was proven and, as far as we are aware of, no proof of the full KAM theorem of 1963 has been
published along these lines~\cite{Bost_KAM} (see also~\cite{Fejoz_KAM}). 
So among the two roads --perturbation theory and Nash-Moser theorems-- none of them seem to give a way out : whichever we choose heavy computations and tricky arguments are waiting for us! 

In fact these two roads merge and Moser's program can be followed together with perturbation theory in the analytic setting. We will develop a framework which surprisingly produces a functional calculus, that we call {\em functorial calculus} which goes beyond the case of Banach algebras. Both an exponential mapping with explicit estimates and a simple convergence criterion emerges from its definition. In fact we shall prove that, like in the finite dimensional situation: {\em the product of exponentials converges if and only if the elements of the Lie algebra we exponentiate are summable for some appropriate norms.}  So the {\em complicated convergence arguments} of Zehnder will disappear.

%%%%%%%%%%%%%%%%%%%%%%%%%%%%%%%%%%%%%%%%%%%%%%%%%%%%%%%%%%%%%%%%%%%
\subsection{Why categories and not just locally convex spaces?}
The answer can already be found in the works of Siegel and Kolmogorov: one has to keep track of how the domains over which the maps are defined~\cite{Kolmogorov_KAM,Siegel_linearisation}.  

The classical way to keep track of these domains is to take disks indexed according to their radii and show that although the radii become smaller they converge to a positive limit. So one easily imagines that by considering more complicated cases, the complexity of computations grows rapidly and like in tensor calculus, we rapidly run into a debauch of indices and estimates for the sizes of the different sets and the norms of the operators which are involved.

The categorical way to handle this problem is to consider indices as if they were the base of a fibration, like coordinates in a topological space.  The fibre over a set may consist of holomorphic functions or vector fields defined over this set and  we think of it as a kind of bundle. Of course, in the background, the sets are still indexed but the categorical setting looks somehow ''coordinate free''. There are many places in mathematics where the operation of ''dropping indices'' involves a simplification of the  setup : Matrices/Linear algebra, Tensors/Differential geometry,  Equations/Schemes.
 
 %%%%%%%%%%%%%%%%%%%%%%%%%%%%
 \subsection{What is new? (or at least seems to be new)}
In the classical case, one considers a single Banach chain indexed by an interval or by integers, the scale is sometimes increasing and sometimes decreasing. This is not the case here: starting from an interval, we construct Banach spaces over different kind of spaces like open triangles, closed triangles and subdomains in $\RM^2$. These play the role of definition domains and the way they are mapped to each other is the key of understanding functorial calculus.

There is another point to be understood. While in the classical case, the norms are constructed a priori, in the categorical context, the norms emerge from general constructions. In particular, these are generally given by complicated formulas that would have been difficult to guess. Note also that unlike the classical approach, the formalism allows us to  work with several scales at the same time. To see how this works in practise, we refer to~\cite{Symplectic_torus,Versal_fields,Herman_conjecture}.

However, our framework is quite restrictive as we content ourselves to the case of linear mappings and therefore, this paper constitutes only the first steps towards a theory of Lie functors, which would go beyond the theory of locally convex Lie groups.

%%%%%%%%%%%%%%%%%%%%%%%%%%%%%%%%%%%%%%%%%%%%%%%%%%%%%%%%%%%%%%%%%%%%%%%%%%%%%
\section*{}
\section{{\large \bf The Functor Category}}
\section*{}

Category theory is a language for a large part of mathematics and we assume the reader has some familiarity with notions from basic category theory and the
most basic categories like {\bf Set}, {\bf Vect} or {\bf Ban} (with bounded linear maps as morphisms).
We refer to  the first chapters of \cite{Mac_Lane_categories} or
\cite{Kashiwara_Schapira_categories} for more details. For other categorical aspects of Banach spaces we refer to \cite{Castillo_Banach_categories}.

\subsection{Small categories}
Recall that a {\em small category} is a category whose objects and morphisms form a set. We
will denote such categories by ordinary capitals like $A, B, C$ and think
of them as sets enriched with an additional structure: for two objects 
$a, a' \in A$ there is specified a set of morphisms $Mor_A(a,a')$, which 
we may think of as arrows running from $a$ to $a'$.
The {\em opposite category} $A^{op}$ of $A$ is defined to have the same 
objects as $A$, but with all arrows reversed:
\[Mor_{A^{op}}(a,a')=Mor_A(a',a).\]

To an ordered set\footnote{We use the french convention. In most English texts, these objects are called {\em partially} ordered set, sometimes abbreviated to {\em poset}.} $(B,\ge)$ is associated a small category with $B$ as set of objects and the set of morphisms from $a$ to $b$ consisting of a single element if $a \ge b$, and empty otherwise. We will not make any distinction between an ordered set and the category it defines. The opposite category of $(B,\ge)$ is obtained by reversing the ordering, we will use the notation $(B,\ge)^{op} =(B,\le)$.  Note that a plain set $B$ can be seen as an example of an ordered set with trivial ordering (two elements can never be compared).\\

A (covariant) functor $\varphi: A \to B$ between two small categories is the 
natural generalisation of a map between sets. 
In case $A$ and $B$ are ordered sets, the functor property of $\varphi:(A,\geq) \to (B,\geq)$ translates into the statement that $\varphi$ preserves the ordering: if $a \ge a'$, then $\varphi(a) \ge \varphi(a')$.
The categorical approach contrasts with the naïve approach were compatibility relations need to be made explicit at every stage.

One can form the category $\mathcal{S}$ of small categories, whose
objects are small categories, and where $Mor_{\mathcal{S}}(A,B)$ consist of the
functors $\varphi:A \to B$.

\subsection{Category of $B$-objects}
We fix a small category $B$.
We denote by $\Sets$ the category of sets with arbitrary maps as
morphisms, $\Vect$ the category of vector spaces over a fixed field 
$K$ with $K$-linear maps as morphisms and $\Ban$ the category of
(real or complex) Banach spaces and continuous linear maps as morphisms. 

\begin{definition}\label{D::relative object}
Let $\mathcal{C}$ be a category and $B$ be a small category.
A $B$-object in $\mathcal{C}$ is a covariant functor
\[ F: B \to \mathcal{C} .\]
\end{definition}

So a $B$-object consist objects $F(b) \in \mathcal{C}$, $b \in B$  
and for each morphism
$\alpha \in Mor_B(b,b')$ a corresponding morphism $F(\alpha) \in Mor_{\Ct}(F(a),F(a'))$, 
such that $$F(\alpha\circ \beta)=F(\alpha) \circ F(\beta),$$ 
i.e. a composition of morphisms in $B$ is mapped by $F$ to a 
corresponding composition of morphisms in $\mathcal{C}$. We call 
$B$-objects in $\mathcal{C}$ also {\em objects over $B$}, 
as one may think of such a functor as a family of objects in $\mathcal{C}$ 
parametrised by $B$.\\ 

For instance a set over the ordered set $(\NM,\ge)$ can be identified with an inverse system of sets and maps
\[\ldots \to X_n \to X_{n-1} \to \ldots \to X_2 \to X_{1} \to  X_0\]
and an object over the opposite category $(\NM,\le)$ can be identified with a directed system of sets
\[X_0 \to X_1 \to X_2 \to \ldots \to X_n \to X_{n+1} \to \ldots\]

In general, if we consider a small category $B$ as a quiver, then a 
$B$-object in $\Vect$ is nothing but a {\em quiver representation}:
for each object of $B$ we are given a vector space, for each arrow 
a corresponding linear map between vector spaces.\\

The $B$-objects of $\mathcal{C}$ form the objects of a new category
$\mathcal{C}_B$; a morphism between functors 
$$F:B \to \Ct,\ G:B \to \Ct$$ is defined as a {\em natural
transformation} between the functors $F$ to $G$. These are defined by morphisms
$$u(a) \in Mor_{\mathcal{C}}(F(a),G(a))$$
for $a \in B$, such that for all $\alpha \in Mor_B(a,b)$ we have a
commutative diagram
$$\xymatrix{F(a) \ar[r]^{u(a)} \ar[d]_{F(\alpha)}& G(a) \ar[d]^{G(\alpha)} \\
 F(b) \ar[r]^{u(b)}& G(b)} 
$$

If $\varphi: A \to B$ is a morphism in $\mathcal{S}$, we can {\em pull-back}
a given $B$-object to an $A$-object $\varphi^*(F):=F \circ \varphi$ 
by composition. 
This natural operation defines a {\em pull-back functor}
\[\varphi^*:\mathcal{C}_B \to \mathcal{C}_A,\;\;F \mapsto F \circ \varphi .\]

%%%%%%%%%%%%%%%%%%%%%%%%%%%%%%%%%%%%%%%%%%%%%%%%%%%%%%%%%%%%%%%%%%%%%%%%%%%
\subsection{Geometric picture}
There is a simple way to associate  to a relative set defined by a 
functor $F:B \to \Sets$ an ordinary map of sets, that may provide some
additional geometrical intuition for the above discussion.
 
\begin{definition}\label{DD::total space} 
Given a $B$-set defined by a functor $F:B \to \Sets$,
we form the disjoint union
\[ X(F):=\bigsqcup_{b \in B} X_b,\;\; X_b:=F(b), \]
and call it the {\em total space} of the $B$-set. 
The map
$$p:X(F) \to B,\ X_b \ni x \mapsto b $$
is called the {\em associated bundle} and the set $X(F)_b=F(b)$ is called the {\em fibre} $p^{-1}(b)$ over $b \in B$.
\end{definition}

We will denote elements of $X(F)$ often by {\em pairs} $(b, x)$, where
$b \in B$ and $x \in X_b$. This suppresses the (crucial) 
information about the maps between the fibres. But these are often 
clear from the context, in which case we often say that $p$ or even $X$ is 
a $B$-set instead of specifying the functor $F$.\\

For a morphism $\p \in Mor_B(a,b)$  there is a corresponding 
map 
\[ \iota_{\p} : X_a \lra X_b\]
that we will call the {\em connecting morphism} of $\p$. If $B$ is an ordered set this means simply that
$a \geq b$ and we may write $\iota_\p=\iota_{ba} $. The functor property translates
into the statements that for all $a,b,c \in B$ with $a \ge b \ge c$.
\[  \iota_{cb}\iota_{ba} =\iota_{ca} .\]
%YY

%%%%%%%%%%%%%%%%%%%%%%%%%%%%%%%%%%%%%%%%%%%%%%%%%%%%%%%%%%%
\subsection{Relative vector spaces}
For a $B$-object in $\Vect$ or $\Ban$, we can form likewise its {total space}
\[ E(F):=\bigsqcup_{b \in B} E_b.\]
which defines a $B$-set $p: E(F) \lra B$. All fibres $E_b=p^{-1}(b)$ of this 
$B$-set now have a vector space or Banach space structure.
 
\begin{figure}[htb!]
\includegraphics[width=0.5\linewidth]{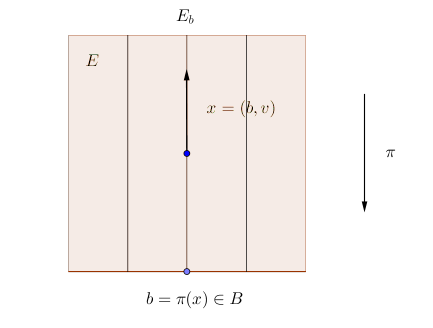}
\end{figure}

As such, $B$-vector spaces have some similarity with the idea
of a vector bundle or the \'etale space of a sheaf. Like for vector bundles, we sometimes identify the point $(b,x)$ of  the total space with the vector $x$ on the fibre.

The fibres over different points 
$b \in B$ can be very different, but the connecting morphism 
$F(\alpha):E_a \to E_b$ lying over a
morphism $\alpha \in Mor_B(a,b)$ gives a way to compare fibres over 
different points (``parallel transport''), and thus provides a 
categorical version of a connection.\ \\

\begin{figure}[htb!]
\includegraphics[width=0.5\linewidth]{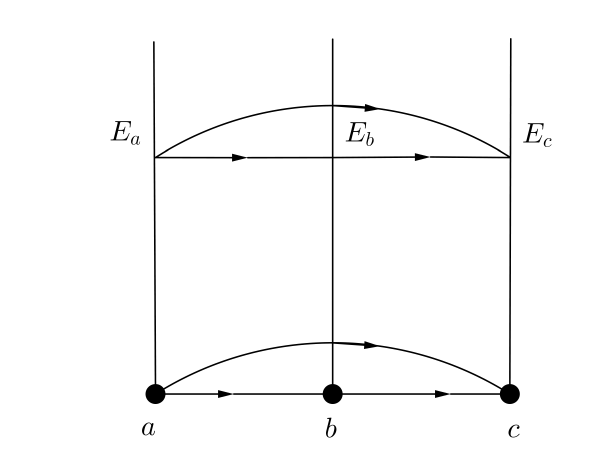}
\end{figure}

%%%%%%%%%%%%%%%%%%%%%%%%%%%%%%%%%%%%%%%%%%%%%%%%%%%%%%%%%%%%%
\subsection{The section functor}

\begin{definition}\label{D::section} Let  $p:X \lra B$ be a set over a small category
 $B$, and $A$ a subset of objects of $B$. Then
{\em a section over $A$} is a map 
$$s:A \lra X$$ such that 
\[p \circ s=Id_A .\]
We denote the set of all sections by $\Gamma(A,X)$.
\end{definition}

So a section just consists of the choice of an element $x_a \in X_a$ for each 
$a \in A$. In case of a relative vector space $p: E \lra B$, we may use
the vector space operations pointwise to define
the structure of a vector space on the set $\Gamma(A,E)$. Note that
\[\Gamma(A,E) = \prod_{a \in A} E_a\]
is just the {\em direct product} of all vector spaces $E_a, a \in A$.
If we define the {\em support} of section $s:A \lra E$ as the 
set of $a \in A$ for which $s(a) \neq 0$, then one can identify the
direct sum \[ \bigoplus_{a \in A} E_a \subset \prod_{a \in A} E_a\]
as the subspace of sections with {\em finite support}.

Note that in the definition of $\Gamma(A,X)$ we ignored the arrows in $B$.
Of course these are usually important.

\begin{definition}\label{D::horizontality}
Let  $p:X \lra B$ be a set over a small category $B$, defined by a
functor $F:  B \to \Sets$ and $A$ a sub-category of $B$. 
A section $s \in \Gamma(A,X)$ is called {\em horizontal} if for all 
$a,b \in A$ and $\alpha \in Mor_A(a,b)$ one has
\[ F(\alpha)(s(a))=s(b) .\]
We denote the subset of horizontal section by
\[ \Gamma^h(A,X) \subset \Gamma(A,X).\]
\end{definition}

%%%%%%%%%%%%%%%%%%%%%%%%%%%%%%%%%%%%%%%%%%%%%%%%%%%%%
\subsection{The set of morphisms between $B$-objects are not $B$-objects}
The contravariant nature of the Hom-functor in the first argument implies
that the morphisms of a category of $B$-objects are, in general, not $B$-objects. Let us clarify  this point by a simple example.

Let $B$ be the category $(\{ 0, 1\},\leq )$. It has two objects $\{ 0, 1\}$ and one non trivial morphism in $Mor(0,1)$.
Let $\Ct=\Sets$ be the category of sets. A functor
$$F:B \to \Ct $$
is nothing else than a map
$$f:X_0 \to X_1. $$
Consider two functors $F,G$ associated to two mappings
$$f:X_0 \to X_1,\ g:Y_0 \to Y_1 $$
We have two associated bundles 
$$X \to B,\ Y \to B $$
each one with two fibres and one connecting morphisms. If $Mor(X,Y)=Mor(F,G)$ were a $B$-object, then it would have also a  connecting morphism:
$$Mor(X_0,Y_0) \to Mor(X_1,Y_1) $$
This means that to any morphism $\p_0 \in Mor(X_0,Y_0)$ can be transported to a morphism $\p_1 \in Mor(X_1,Y_1)$ in such a way that the following diagram commutes:
$$\xymatrix{ X_0 \ar[r]^-f \ar[d]_-{\p_0} & X_1 \ar[d]^-{\p_1}\\
Y_0 \ar[r]^-g & Y_1}
$$
But this will not be possible in general. Take for instance 
$$X_0=X_1=Y_0=Y_1=\{-1,1\}.$$ If $f$ is a constant map and $g=\p_0=id$ then $g \circ \p_0=\text{id},\ \p_1 \circ f=\text{ constant}$.
Hence we find a contradiction.

\begin{figure}[htb!]
\includegraphics[width=0.7\linewidth]{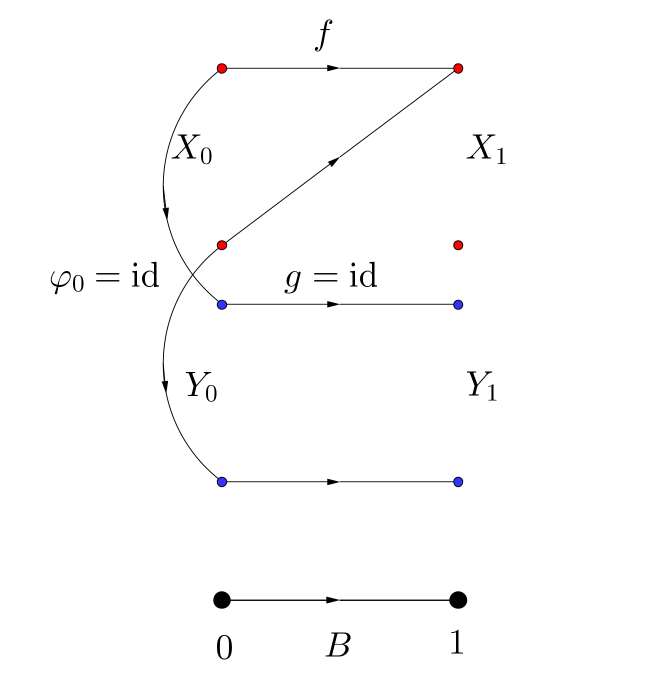}
\end{figure}

In the theory of normal forms, we often deal with morphisms of a category and sometimes with morphisms of morphisms (say if we map a vector field to another), therefore the category of $B$-objects cannot be taken as our fundamental construction.
%%%%%%%%%%%%%%%%%%%%%%%%%%%%%%%%%%%%%%%%%%%%%%%%%%%%%%%%%%%%%%%%%%%%%%%%%%%%%%%%%%%%%%%%%%%
\subsection{Construction of the functor category}
We fix a category $\mathcal{C}$ and consider $X \to B$ and $Y \to B$ two $\Ct$-objects over $B$, defined by functors
$$F,G:B \to \Ct.$$ 
We  observed that the set of morphisms between them does not form a $B$-set
in a natural way. There is no obvious way to associate to $\alpha \in Mor_B(a,b)$ a morphism 
$$u_\a:X_a=F(a) \to Y_b=G(b).$$
To adress this problem, it turns out that instead of working over
a fixed category $B$, it is useful to consider consider the category of {\em all the relative objects of $\mathcal{C}$ over all
possible small categories $B$}. This defines a big category $\mathcal{C}_*$ that we call the {\em functor category} of $ \Ct$ which we will now spell out.

Given two functors
$$F:B \to \mathcal{C},\ G:C \to \mathcal{C},$$  
we define a set of morphisms which is
a $B^{op} \times C$-object in $\mathcal{S}$. 
Here $B^{op}$ stands for the opposite category obtained by reversing
the arrows. 
 
Take two pairs $(a,b)$ and $(c,d)$ inside $B^{op} \times C$.
A morphism $(\alpha,\beta) \in Mor_{B^{op} \times C}((a,b),(c,d))$ 
is defined as a pair of morphisms
\[\alpha \in Mor_{B^{op}}(a,c)=Mor_B(c,a), \;\;\beta \in Mor_C(b,d).\]
If $f \in Mor_\Ct(F(a), G(b))$ then we obtain by composition 
\[ G(\beta) \circ f\circ F(\alpha) \in Mor_{\Ct}(F(c),G(d)) .\]

In terms of the associated bundles, we start with two bundles 
$$X \to B,\ Y \to C $$
and construct the  bundle
$$Mor(X,Y) \to B^{op} \times C  $$
with fibre $Mor(X_a,Y_b)$. If we now consider $(\alpha,\beta) \in Mor_{B^{op} \times C}((a,b),(c,d))$
then it maps the morphism $f=u(a,b) \in Mor(X_a,Y_b)$ to $g=u(c,d) \in Mor(X_c,Y_d)$ by the commutative diagram
$$\xymatrix{X_a \ar[r]^-{f} & Y_b \ar[d]^{\b} \\
X_c \ar[u]-^{\a} \ar[r]^-{g} & Y_d}$$

\begin{definition}
 The functor category of  a category $\Ct$ denoted $\Ct_*$ has objects functors from a small category to $\Ct$. A morphism between functors 
$$F:B \to \Ct,\ G:C \to \Ct$$ is defined by a collection of morphisms
$$u(a,b) \in Mor_{\mathcal{C}}(F(a),G(b))$$
for $a \in B^{op},\ b \in C$, such that for all $(\alpha,\beta) \in Mor_{B^{op} \times C}((a,b),(c,d))$ 
we have a commutative diagram
$$\xymatrix{F(a) \ar[r]^{u(a,b)} & G(b) \ar[d]^{G(\beta)} \\
 F(c) \ar[u]^{F(\alpha)}  \ar[r]^{u(c,d)}& G(d)} 
$$
\end{definition}
We denote this set of morphisms by
$$Mor_{\Ct}(F,G).$$
Note that this is naturally a set over $B^{op} \times C$.

The category $\mathcal{C}_*$ comes 
with a functor 
\[ \mathcal{C}_* \to \mathcal{S},\]
which maps each relative object to its base $B \in \mathcal{S}$. The
'fibre' over $B$ is the category $\mathcal{C}_B$ of $B$-objects, exhibiting
$\mathcal{C}_*$ as a fibred category over $\mathcal{S}$.  

\begin{example} 
Consider again the category $B=(\{ 0, 1\},\leq )$ with two objects $\{ 0, 1\}$ and one non trivial morphism in $Mor(0,1)$.
Take two functors
$$F,G:B \to \Ct $$
associated to two maps
$$f:X_0 \to X_1,\ g:Y_0 \to Y_1. $$
The morphisms form a $B^{op} \times B$ object:
$$Mor(X,Y) \to B^{op} \times B $$
The category $B^{op} \times B$ has four elements and the arrows encode how morphisms can be composed:\\
\vskip0.2cm

\begin{figure}[htb!]
\includegraphics[width=0.5\linewidth]{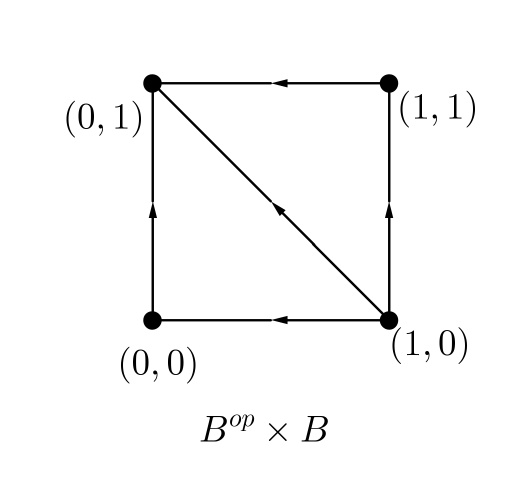}
\end{figure}
So starting with a quiver of two points connected by an arrow, we end up with a much more elaborated quiver.
\end{example}

\begin{example}
  Let $$D_t:=\{ z \in \CM\;|\;|z| <t\}$$
  denote the open disc of radius $t$ in the
complex plane. There are obvious inclusion maps $$D_s \hookrightarrow D_t$$
for $s \le t$, so we can see $D_t$ as fibre of a relative set $D$ over
$(\RM_{>0},\le)$.

Let $\Ot(D)$ be the vector space over $(\RM_{>0},\ge)$ with
fibre $$\Ot(D)_t:=\Ot(D_t):=\{f:D_t \to \CM\;|\;f\;\;\textup{holomorphic}\}$$
the space of holomorphic functions on $D_t$. The connection maps for $t \ge s$
\[ \iota_{st}:\Ot(D)_t \lra \Ot(D)_s\]
are the restriction maps. Note that these connection maps represent a {\em
  horizontal section $\iota$} of the relative vector space
  $$\Hom(\Ot(D),\Ot(D)) \to \RM_{>0}^{op} \times \RM_{>0} $$
  over the closed simplex
  \[\overline{\Delta}:=\{(s,t) \in \RM_{>}\times \RM_{>} \;|\;s \le t\}\]
  
\begin{figure}[htb!]
 \includegraphics[width=0.3\linewidth]{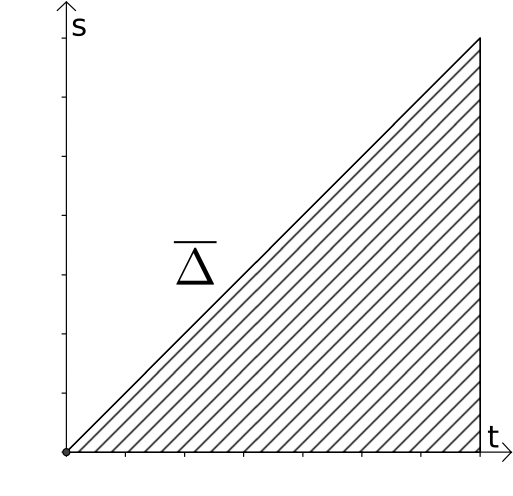}
\end{figure}

Now consider the scaling map 
$$m:\CM \to \CM,\ z \mapsto 2z$$ It induces
a well-defined map 
$$\CM\{ z \} \to \CM\{ z \},\ f(z) \mapsto f(2z)$$ on {\em germs} of holomorphic functions
at $0$, but it does not act on the fibres of $\Ot(D)$. Rather, it
produces a {horizontal section} $\s_m$ of $\Hom(\Ot(D),\Ot(D))$
over the set $$A:=\{ (t,s)\;|\; 2s \leq t \} .$$
Indeed, for  $(t,s) \in A$ the maps
\[ \Ot(D_t) \to \Ot(D_s);\;f(z) \mapsto f(2z) \]
is well defined and are compatible with restrictions.

\begin{figure}[htb!]
 \includegraphics[width=0.3\linewidth]{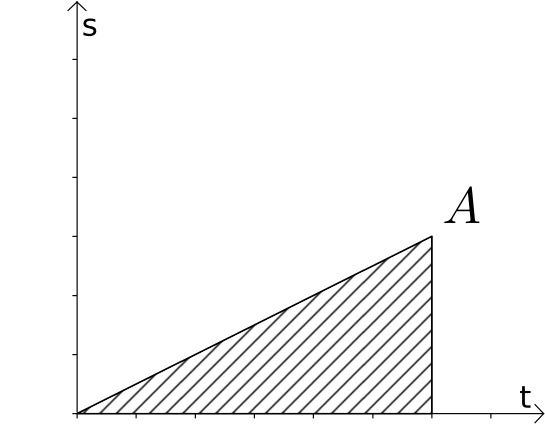}
\end{figure}

The set $A$ could be called the definition domain of our map $m$. We see that even for such
a basic example, starting with $\RM_{>0}$ as parameter space, we immediately
end up with  objects defined over a triangular region of $\RM_{>0}^2$. 
\end{example}
  
The construction of functor categories leads to the definition of Hom-spaces in a  natural way. Many considerations have a natural geometrical intepretation in the parameter space, as we will see in this paper. So if we use as heuristic rule that all constructions should be functorial, then this limits considerably the number of possibilities to define basic concepts; in turn these concepts organise the computations and we may use them as Ariadne's thread to get out of the labyrinth of indices.

%\section*{}

\section{{\large \bf Relative Banach spaces}}
%\section*{}
 %\setcounter{section}{2}
 %\setcounter{subsection}{0}
%%%%%%%%%%%%%%%%%%%%%%%%%%%%%%%%%%%%%%%%%%%%%%%%%%%%%%%%%%%%%
\subsection{Definition}
Recall that we denote by $\Ban$ the category with objects Banach spaces 
and morphisms continuous (=bounded) linear mappings and that we can form 
its functor category $\Ban_*$. An element of this category is nothing but a functor 
$$F:B \to \Ban $$ to 
the category of Banach spaces where $B$ is a small category. We formed the associated 
total space 
\[ E=\bigsqcup_{b \in B} E_b,\;\;\; E_b=F(b), \]
with its associated bundle
\[ p: E \to B.\]
We will now discuss some particular structures associated to this
situation.

For relative Banach spaces $E$ over $B$ and $F$ over $C$, 
use the the notation 
$$Hom(E,F) \to B^{op} \times C$$ 
for the relative Banach space whose fibre over $(a,b)$ is the Banach space 
$\Hom(E_a,F_b)=Mor_{\bf Ban}(E_a,F_b)$ of (continuous) linear maps $E_a \to F_b$ 
with the operator norm.

\subsection{Norms matter}
By definition, each Banach space $E_b$ is equipped with a specific 
norm $|-|_b$. For each vector  $x=(b,v) \in E$ we define
\[ |x|=|v|_b \in \RM_{\ge 0}, \]
and so we obtain a map
\[ |-|: E \to \RM_{\ge 0},\;\;x \mapsto |x|.\]

Together with the projection $p: E \to B$ to the base we obtain the 
{\em norm map}\index{norm map}
$$\nu:E \to B \times \RM_{\ge 0},\ x=(b,v) \mapsto (b,|x|)=(b, |v|_b) . $$

If $s \in \Gamma(A,E)$ is a section, then for each $a \in A$ we can
consider the norm $|s(a)|$, hence we obtain a {\em norm function}\index{norm function}
\[ |s|:=|-|\circ s: A \lra \RM_{\ge 0}, \;\;a \mapsto |s(a)|\]
of the section.

\subsection{Banach spaces of sections } 
\label{SS::bounded}
In perfect analogy with vector bundles equipped with a Riemannian metric, we consider the spaces of bounded sections and horizontal bounded
sections.
\begin{definition}\label{D::bounded section} Let $E$ be a Banach space over $B$ and $A$ subset of 
objects of $B$. A section $s \in \Gamma(A,E)$ is called {\em bounded},
if the norm-function 
\[ a \mapsto |s(a)|\]
is bounded on $A$. We use the notation
\[ \Gamma^{b}(A,E):=\{ s \in \Gamma(A,E)\;\;|\;\;\sup_{a \in A} |s(a)| < \infty\} \]
for the set of all bounded sections. For $s \in \Gamma^{b}(A,E)$ we put
\[ |s|_A:=\sup_{a \in A} |s(a)| .\]
\end{definition}
Although elementary the following proposition is  fundamental:
\begin{proposition}
If $E$ is a Banach space over $B$ and $A$ a subset of objects then $\Gamma^b(A,E)$ with $|-|_A$ is a 
Banach space.
\end{proposition}
\begin{proof} We write  temporarily $|-|$ for $|-|_A$. Clearly, one has:
\begin{align*}
|s_1+s_2| &\le |s_1|+|s_2|,\\
|\lambda\cdot s|&=|\lambda| \cdot |s| \end{align*}
and so $\Gamma^b(A,E)$ is a normed vector space.
Now consider a Cauchy sequence of sections $s_n$ in $\Gamma^b(A,E)$, so for any $\epsilon >0$
we can find $N \in \NM$ such that
\[ |s_n-s_m| \le \epsilon \;\;\;\textup{for}\;\;n,m \ge N .\]
As for all $a \in A$ one has
\[ |s_n(a)-s_m(a)| \le |s_n-s_m|,\]
it follows immediately that for any $a \in A$ the sequences $s_n(a)$ are 
Cauchy sequences in $E_a$ and thus converge in $E_a$ to an element $s(a)$ 
which defines a section
$$s \in \G(A,E) .$$ 
As 
\[||s_n|-|s_m|| \le |s_n-s_m|, \]
the sequence of norms $(| s_n |)$ is a Cauchy sequence in $\RM$ and is 
therefore bounded. As a consequence, the norm function $|s|$ is also 
bounded and thus the limit $s$ is again an element in $\Gamma^b(A,E)$.
\end{proof}

\begin{proposition}
\label{P::horizontal}
If $E$ is a Banach space over $B$ and $A \subset B$ a subcategory then the space 
\[ \G^h(A,E) \cap \G^{b}(A,E)\]
of horizontal and bounded sections is a Banach space.
\end{proposition}
\begin{proof}
The relative Banach space is defined by a functor
$$F:B \to \Ban $$
Consider a morphism  $\alpha \in Mor_A(a,b)$, the equations
\[F(\alpha) ( s(a) ) = s(b),\;\;\;\forall a, b \in A,\]
defining horizontality, determine a closed linear subspace of $\Gamma^{b}(A,E)$.
\end{proof}

\begin{definition}\label{D::horizontal bounded section} 
  Let $E$ be a Banach space over $B$. The Banach space of horizontal bounded
  sections over $A \subset B$ is denoted by 
\[ \G^{\infty}(A,E):=\G^h(A,E) \cap \G^{b}(A,E) .\]
\end{definition}

\subsection{Rescaling}

Rescaling  the norm in a single Banach space is usually of no 
great importance. In the relative situation, the rescaling operation becomes very non-trivial,
as the rescaling may depend on the point $b \in B$ in question. For instance if we the derivative
for bounded holomorphic functions on a disk, then by rescaling  the norm can be made bounded.

\begin{definition}\label{D::rescaling} 
A function $\lambda: B \to \RM_{>0}$ is called a {\em weight function}.
If $E \to B$ is a Banach space over $B$ and $\lambda$ a weight function, 
we define the {\em rescaled Banach space}\index{rescaled Banach space} 
over $B$ to be
\[ \widetilde{E} \to B \]
with 
\[ \widetilde{E}_b = E_b, \textup{and norm}\;\;\; |-|^{(\lambda)}:=\lambda(b) \cdot |-|_b .\]
\end{definition}

Note that the operation of rescaling has  some formal resemblance  
with the operation of tensoring with a line bundle in algebraic geometry.

\subsection{The derivative (part I)}
From now on, we will use systematically the derivative to explain the various constructions leading to the notion of local operators.
We consider again the relative disc 
$$D \to (\RM_{>0},\le)$$ 
with fibre $D_t=\{ z \in \CM:|z|<t \}.$ Given an open set $U \subset \CM$, let us denote by $\Ot^b(U)$ the Banach space of holomorphic bounded functions on $U$ with supremum norm
$$|f|=\sup_{z \in U} |f(z)|. $$
We have a relative Banach space
$$E:=\Ot^b(D) \to \RM_{>0} $$
with fibre $\Ot^b(D_t)$. For $t > s$, there are maps
\[ u_{st}: \Ot^b(D_t) \to \Ot^b(D_s),\;\;f \mapsto \frac{df}{dz}\]
that are compatible under restrictions. This means that
the derivative is a global section
\[ \frac{d}{dz} \in \G^{h}(\Delta,Hom(E,E)) \]
over the simplex
$$\D=\{ (t,s) \in \RM^2_{>0}: s<t \} .$$
However this section is not bounded, as for instance the function
$$f(z)=(1-z)\log(1-z) \in \Ot^b(D_t),\ t=1 $$
defines a unique horizontal section that we denote by:
$$\s=(1-z)\log(1-z) \in \G^h(]0,1],\Ot^b(D)) $$ 
Its derivative defines a section only over $]0,1[$:
$$\s'=-1-\log(1-z) \in \G(]0,1[,\Ot^b(D))$$
and for $s$ close enough to $1$ we have
$$\s'(s)=-1-\log(1-s) \xrightarrow[s \to 0]{} +\infty $$
However one has the {\em Cauchy inequality}
\[ |\frac{d}{dz} f|_s \le \frac{1}{t-s}|f|_t \]
as a direct consequence of Cauchy's integral formula.
Therefore the derivative is an unbounded section over the simplex $\D$
with a pole along the diagonal. This is the typical behaviour of the
norm for a differential operator of order $k$: the norm has a pole of order $k$
along the diagonal. Now rescaling the space of operators by the function $\l(t,s)=t-s$
transforms the unbounded section $d/dz$ into a bounded section in the rescaled Banach
space. 
%%%%%%%%%%%%%%%%%%%%%%%%
\section{Kolmogorov spaces}
%%%%%%%%%%%%%%%%%%%%%%%%%%%%%%%%%%%%%%%%%%%%%%%%%%%%%%%%%%%%%%%%%%%%%%%%%%%%  
\subsection{Kolmogorov spaces}
We will now look at the most important case of Banach spaces $E$ over a basis $B$.
\begin{definition}\label{D::Kolmogorov space}  A Kolmogorov space\footnote{In set-theoretical topology the term Kolmogorov space is used for topological
spaces for which the $T_0$-separation axiom holds. We will never use the notion in this sense, although the topology defined by downsets in an ordered set
provide natural examples such $T_0$ spaces.}\index{Kolmogorov space} 
is a Banach space $E$ over a small category $B$, such that for all $\p \in Mor_B(a,b)$
{\em connecting morphisms}
$$\iota_{\p} : E_a \to E_b, $$
have norm at most $1$, so that
\[ |\iota_{\p}(x)| \le |x| . \]
We denote by $\Kol_*$ the full subcategory of $\Ban_*$ consisting of
Kolmogorov spaces.
\end{definition}
The relative Banach space $\Ot^b(D) \to \RM_{>0}$ and $\Hom(\Ot^b(D),\Ot^b(D))$ are obvious examples of Kolmogorov spaces over $\RM_{>0}$ and $\RM_{>0}^2$ respectively.
%%\vskip0.3cm
%\begin{figure}[htb!]
%\label{F::K2}
%\includegraphics[height=5cm]{k2.jpg}
%\end{figure}
%\vskip0.3cm 
%%%%%%%%%%%%%%%%%%%%%%%%%%%%%%%%%%%%%%%%%%%%%%%%%%%%%%%%%%%%%%%%%%

\subsection{Rescaled Kolmogorov spaces}
 By rescaling a Kolmogorov space $E \to B$ by a weight function $\l$ that
is increasing, we get a new Kolmogorov space 
$$E[\l] \to B $$
with fibres $(E_b,\l(b)|_b\cdot |)$. 
However in many applications we encounter situations where $\l$ is 
not increasing  and so we just get an ordinary relative Banach space:
$$ \widetilde{E} \to B. $$
 
 This is already the case for the derivative where we rescaled the relative space
$$\Hom(\Ot^b(D),\Ot^b(D)) \to \D^{op} \times \D $$
by the function $\l(t,s)=t-s$. This rescaled space is NOT a Kolmogorov space. Indeed, the connecting morphisms (= the restriction mappings)
$$\p_{ts}:\Ot(D)_t \to \Ot(D)_s,\ f \mapsto f $$
have now norm $t-s$. So the connecting morphisms
$$\Hom(\Ot^b(D),\Ot^b(D))_{t,s} \to \Hom(\Ot^b(D),\Ot^b(D))_{t',s'} $$
are no longer of norm $ \leq 1$.

In practise, this means that we lose completely our control on estimates. So we must find a way to construct a Kolmogorov space out of a relative Banach space and this is the Kolmogorification procedure. From a theoretical point of view, the construction might seem tautological but when we unwind the definition in concrete examples, we often end-up with complicated formulas for the norms which are constructed.

%%%%%%%%%%%%%%%%%%%%%%%%%%%%%%%%%%%%%%%%%%%%%%%%%%%%%%%%%%%%%%%%%%%
\subsection{Kolmogorification}

For an ordered set $(B,\ge)$, the down-set of $b$ consist of all $b' \in B$ for which
$b \ge b'$. Similarly we define in general

\begin{definition}\label{D::up-set}\label{D::down-set} 
Let $B$ be a small category and $b \in B$.
We define the {\em down-set} of $b$ as
\[ ]-\infty,b] :=\{ b' \in B\;|\; \exists \p \in  Mor_B(b,b')\},\]
and similarly the {\em up-set}
 \[ [b,+\infty[:=\{ b' \in B\;|\; \exists \p \in  Mor_B(b',b)\}.\]
 
\end{definition}
 
Kolmogorov spaces have the following characteristic property:

\begin{proposition} \label{P::tauto} Let $E$ be a Kolmogorov space
over $B$ and $b \in B$. Then there is a natural isometry of Banach spaces
\[ E_{b} = \Gamma^{\infty}(]-\infty,b],E) .\]
\end{proposition}

\begin{proof}
A vector $v \in E_{b}$ determines a unique horizontal section 
$$s:]-\infty,b] \to E,\ \p \mapsto \iota_{\p}(v) \in E_{\p(b)}$$
The Kolmogorov space property tells us that
$$|\iota_{\p}(v)| \le |v|.$$
Therefore the norm function $|s(\p)|$ attains its maximum at $\p=id$.  
\end{proof}

To any Banach space $E$ over $B$ we can associate in a natural
way a Kolmogorov space $K[E]$ over $B$, by setting
\[ K[E]_b :=\Gamma^{\infty}(]-\infty,b],E), \]
which is the space of {\em bounded horizontal sections} of $E$ over the
{\em down-set} of $b$.

If $s$ is such a section, we assign to it the norm
\[ |s|:=\sup_{\p \in ]-\infty,b] } |s(\p)| .\]  
Clearly, if $\p in Mor(b,b')$ then by composition of morphisms with $\p$,
we have $]-\infty,b'] \subset ]-\infty,b]$, so that
the connecting morphisms
$$K[E]_b:=\Gamma^{\infty}(]-\infty,b],E) \to K[E]_{b'}:=\Gamma^{\infty}(]-\infty,b'],E) $$
have indeed norm at most one, as we are taking the supremum over a smaller set.
As the space of horizontal bounded sections is a Banach space (Proposition \ref{P::horizontal}),
we get that:
$$K[E] \to B $$
is a Kolmogorov space.

\begin{definition}
 The space $K[E] \to B$ is called the {\em Kolmogorification} of the relative Banach space
$E \to B$.
\end{definition}
  If the vector space $E$ is already Kolmogorov, then
$K[E] = E$, and we get back the original space. Note that the
construction is functorial, so we obtain a Kolmogorification functor
\[\Ban_* \to \Kol_*,\;\;\; E \to K[E]\]
The process of Kolmogorification has a similarity to the 
sheafification of a presheaf.
%%%%%%%%%%%%%%%%%%%%%%%%%%%%%%%%%%%%%%%%%%%%%%%%%%%%%%%
\subsection{Rescaled Kolmogorov spaces}
As we observed if we rescale a Kolmogorov space $E \to B$ by a weight function $\l$ then, as a general rule, the rescaled space 
$$ \widetilde{E} \to B $$
with fibres $(E_b,\l(b)|_b\cdot |)$ is not a Kolmogorov space but we still can form the Kolmogorification for which we will use the notation:
\[ E[\l]:=K[\widetilde{E}].\]
\begin{example}
  Consider again the Kolmogorov space $\Ot^b(D)$ over $(\RM_{>0},\ge)$
  of bounded holomorphic functions on the relative disc $D$.
  One can form the Kolmogorov space, rescaled by $\l=t-s$:
 $$\Hom(\Ot^b(D),\Ot^b(D))[\l] \to \D^{op} \times \D .$$
 Then the derivative $\frac{d}{dt}$ can be considered as bounded
 global section of this space.
\end{example}
\begin{example}
 Let 
 $$E \to \RM_{>0}$$
 be a Kolmogorov space and rescale it by the function $\l=t^k$, we get a new Kolmogorov space $\Mt^k(E)$ which generalises the $k$-power of the maximal ideal of a local ring. Indeed if we consider the case $\Ot^b(D)$ then we precisely get the Kolmogorov spaces
 whose fibres consists of holomorphic with first $(k-1)$ vanishing derivatives at the origin.
\end{example}

%\begin{definition} If $E \to B$ is a Kolmogorov space over $B$ and 
%$\lambda: B \to \RM_{>0}$ a function, we define the {\em rescaled 
%Kolmogorov space} 
%\[ E (\l) \to B \]
%to be the Kolmogorification of relative Banach space
%\[ \widehat E (\l) \to B .\]
%\end{definition}
%%%%%%%%%%%%%%%%%%%%%%%%%%%%%%%%%%%%%%%%%%%%%%%%%%%%%%%%%%%%%%%%%%%%%
\subsection{Opposite Kolmogorov space}\index{opposite Kolmogorov space}
Still we have not finished with the derivative. If $E,F$ are relative Banach spaces over $B$ then we can form
the relative Banach space
$$\Hom(E,F) \to  B^{op} \times B $$
and if $(\p,a,b)$ is an element of $\Hom(E,F)$ we may associate a horizontal section to it on the down set of $(a,b)$. Differential operators have one peculiarity: they define sections also on upsets. 

Say if we consider again the example of $E=\Ot^b(D)$ and if $\p:=\a(z)\d_z$ and $\a \in E_t$, that is $\a$ is bounded and holomorphic
inside $D_t$. Then for any $s \leq s' \leq t' \leq t$, the map
$$E_{t'} \to E_{s'},\ f \mapsto a(z)f'(z) $$
is defined. While ordinary morphisms are defined over $B^{op} \times B$, differential operators have the peculiarity to  define sections of a relative Banach space with the opposite base $  B \times B^{op} $. 

Obviously, by this operation downsets and upsets
get interchanged. Given a Banach space 
$$E  \to B$$ there is the opposite Kolmogorification procedure. For $b \in B$ we can define a Banach space by taking 
bounded horizontal sections over the {\em upset}:
$$ E^{op}_b:=\G^\infty([b,+\infty[,E),$$ 
For  $b' \ge b$ there is a natural connecting morphism
\[ E_b^{op} \to E_{b'}^{op}\]
with norm $\le 1$. So the $E^{op}_b$ form in a natural way
a Kolmogorov space
$$E^{op} \to B^{op}$$
called {\em opposite Kolmogorov space}.
%%%%%%%%%%%%%%%%%%%%%%%%%%%%%%%%%%%%%%%%%%%%%%%%%%%%%%%%%%%%%%%%%%%%%%

\section*{}
\section{{\large \bf  Functorial calculus}}
\section*{}
%\setcounter{section}{3}
%\setcounter{subsection}{0}
 
%%%%%%%%%%%%%%%%%%%%%%%%%%%%%%%%%%%%%%%%%%%%%%%%%%%%%%%%%%%%%%%%%%%%%%%%%%%%%%
Like the well-known functional calculus in Banach spaces, we would like to
develop something similar in the context of Kolmogorov spaces, which we will
call {\em functorial calculus}. In contradistinction to functional calculus,
it turns out that for the moment a satisfactory functorial calculus can only
defined for a special class of operators, that we call {\em first order local operators}.
We will first define this class of operators and, to do this, we start
with two simple fundamental examples. The class we consider is much more general than what is usually understood by local operators (operators involving a finite number of derivatives). It includes for instance changes of scales, Hadamard products and division by holomorphic functions. 
%%%%%%%%%%%%%%%%%%%%%%%%%%%%%%%%%%%%%%%%%%%%%%%%%%%%%%%
\subsection{The derivative(part II)}
Let us consider again the Kolmogorov space $E:=\Ot^b(D)$ over $\RM_{>0}$ of
bounded holomorphic functions on the relative disk. We have seen that the derivative $\frac{d}{dz}$ is not bounded as section of $\Hom(E,E)$, but will
be so as a section of
$$\Hom(E,E)^{op}[\l] \to \D \times \D^{op}$$
where we rescale the norm 
on the fibre $\Hom(E_t,E_s)$ by a factor $\l(t,s)=t-s$ and the Cauchy inequality implies that $|\frac{d}{dz}|=1$.

%%%%%%%%%%%%%%%%%%%%%%%%%%%%%%%%%%%%%%%%%%%%%%%%%%%%%%%%%%%%%%%%%%%%%%%%%%%%%%%%%%%
\subsection{Division of holomorphic functions}
For the second example, we consider the same Kolmogorov space 
$$E:=\Ot^b(D) \to \RM_{>0}$$ 
and let $M \subset E$ be the subspace of functions $f$ with $f(0)=0$. 

We can divide a function $f \in M_t$ by $z$ and form the holomorphic 
function $g(z):=f(z)/z \in E_t$ and we have an estimate
\[ |g| \le \frac{1}{t}|f| .\]
As suggested by the above example of the derivative, it is natural to
expect that the categorical approach tells us to consider division also
as a two parameter family of operators $u_{st} \in \Hom(M_t,E_s)$:
\[ u_{st}: M_t \to E_s, \ f \mapsto \frac{1}{z}f\]
parametrised by a closed simplex
\[  \overline{\Delta}:=\{ (t,s) \;\;| t  \geq s \}.\]
Is the section bounded? Of course not. If we consider the horizontal section 
$$\s=z \in \G(]0,t],M)$$
then
$$|\s(t)|=t,\ |\frac{1}{z}\s(t)|=1  $$
so the operator norm of the map
$$u_{st}:M_t \to \Ot^b(D)_s, f \mapsto \frac{1}{z} $$
equals $1/t$. This means that the horizontal section 
$$u \in \G^h(\overline{\Delta},\Hom(M,E))$$
has a simple pole along the coordinate axis $t=0$. We may however rescale the norm by a factor $\l(t,s)=t$ and then division operator becomes an element of the Banach space:
\[ \G^{\infty}(\overline{\Delta},\mathcal{H}om(E,E)[\l]). \] 

Back in 1944, H. Cartan showed that a similar phenomenon appears for division of systems of
holomorphic functions where one gets the same type of estimate~\cite{Cartan_ideaux}. So when performing divisions or when solving equations, we will, as a general rule, rescale by functions vanishing on the boundary of a simplex.

Sometimes instead of dividing by a function we multiply by some function and therefore we rescale by $s^k$
instead of dividing by $t^k$. If we combine these three types of operations (multiplication, division and differentiation), we encounter operators $f \mapsto g$
with estimates of the form 
\[ |g|_s \le \frac{s^l}{t^k(t-s)^m} |f|_t .\]
So we will rescale the norms by factor of the form $s^{-l}t^k(t-s)^m$.
The theory of {\em local operators} is an abstraction and 
generalisation of these examples. 

\subsection{Local operators}
We now consider Kolmogorov spaces $E$ and $F$ over a base $B$. As before, we write
$$\D:=\{(t,s) \in B^{op} \times B:t>s \} .$$

From a categorical point of view the set $\D$ is to be considered as lying in $B^{op} \times B$ and therefore the condition
$$(t',s') \leq (t,s) $$
corresponds to the inequalities $t' \leq t$ and $s' \geq s$.

The opposite Kolmogorov space $\Hom(E,F)^{op}$ restricted to $\D$ can be
rescaled by a weight function: 
$$\l:B \times B \to \RM_{>0}$$
and we get a resulting Kolmogorov space
over $\D^{op}$:
$$\Hom(E,F)^{op}[\l] \to \D^{op},$$
whose fibre above $(t,s)$ consists of bounded horizontal section 
$u \in \G^\infty(\D_{s,t},\Hom(E,F))$ equipped with the supremum norm. 

\begin{definition}
The Kolmogorov space 
$$\Lt^{\l}(E,F):=\Hom(E,F)^{op}[\l] \to \D^{op} ,\;\; \D^{op} \subset B \times B^{op}$$
is called {\em the Kolmogorov space of $\l$-local operators}.
\end{definition}
When we unwind the definition, we have as fibre
\[\Lt^{\l}(E,F)_{t,s}=\G^{\infty}(\D(t,s),Hom(E,F)[\l]),\]
the Banach space of bounded horizontal sections over the triangle
\[ \D(t,s)=\{(t',s') \in \D\;|\;t' \le t, s' \ge s \}.\]
\begin{center}
\begin{tikzpicture}[scale=2.5]
	\draw[->] (0,0) -- (0,1) node[left] {$s$};
	\draw[->] (0,0) -- (1,0) node[below] {$t$};
	\draw[thin, domain= 0.0:1.0] plot (\x, {\x)});
	\draw[thick, pattern= north east lines] (0.25,0.25)-- (0.75,0.25) -- (0.75,0.75) -- cycle;
\end{tikzpicture}
\end{center}
Partial differential operators are local but there are also many other operators which are local and are not of this sort. Consider for instance the Kolmogorov spaces 
 $$\Ot^b(D) \to \RM_{>0},\ \Ot^h(D) \to \RM_{>0}  $$
 of bounded holomorphic functions
or $L^2$ holomorphic  functions respectively. The injection 
$$\Ot^h(D) \to \Ot^b(D)$$ is not a usual Banach space mapping between the fibres (an $L^2$ function is not necessarily bounded) but it is a local operator~(we refer to \cite{Symplectic_torus} for more details).

Our formalism provides an automatic method of handling compatible 
families of operators, and packing them into Banach spaces with appropriate 
norms in a systematic way.

%%%%%%%%%%%%%%%%%%%%%%%%%%%%%%%%%%%%%%%%%%%%%%%%%%%%%%%%%%%%%%%%%%%%%%%%%%
\subsection{Composition of local operators}
If $E,F,G$ are Kolmogorov spaces over $B$, we have a well 
defined bilinear composition maps
\begin{align*}
\circ: \Hom_{\D(t,s)} (E,F) \times \Hom_{\D(t,s)} (F,G) &\to \Hom_{\D(t,s)}(E,G)\\
 (u,v)& \mapsto v \circ u 
\end{align*}
where we used the notation
$$\Hom_{A} (-,-) := \G(A,\Hom(-,-))$$
This composition is defined as follows: given $u \in \Hom_{\D(t,s)} (E,F)$ and $v \in \Hom_{\D(t,s)} (F,G)$
and $(p,q) \in \D(t,s)$, we set
\[ (v \circ u)_{(p,q)}:=v_{p,r}\circ u_{r,q},\]
where $ p < r < q$. It is easily checked, that by the horizontality of $u$ and $v$ 
{\em the result does not depend on the choice of $r$}.
Obviously, if $u \in \Lt^{\l}(E,F) $ and $v \in \Lt^{\mu}(F,G)$ then
\[ v \circ u \in \Lt^{\rho}(E,G)\] 
if one has
\[ \rho(t,s) \le \sup_{t \ge m \ge s}(\l(t,m)  \mu(m, s) )\]

\begin{definition}
A sequence of weight functions $(\l_1,\l_2,\ldots,\l_k,\ldots)$ is called
{\em submultiplicative} if  for all $p,q$ there exist an $ m \in [s,t] \subset B$ with
\[ \l_{p+q}(t,s) \le \l_p(t,m)\l_q(m,s).\]
\end{definition}

Let us consider the special case $B=\RM_{>0}$.
One verifies that 
\[ \l_k(t,s) :=(t-s)^k/k^k,\ s,t \in \RM_{>0}\]
is a submultiplicative weight sequence: for given $s < t$ the maximum
of the function 
$m \mapsto (t-m)^p (m-s)^q$ on the interval $[s,t]$ is attained at the point
\[  m=\frac{p}{p+q}s+\frac{q}{p+q}t,\]

%\begin{definition} Fix a submultiplicative function $\l$. Then the space
%\[\Lt^{k}(E,F):=\Lt^{\l_k}(E,F) \]
%is called {\em the Kolmogorov space of $k$-local operators} (w.r.t. $\l$).
%\end{definition}

Local operators behave well with respect to composition:
 
\begin{proposition}
\label{P::algebra}
If the sequence $(\l_k)$ is submultiplicative then for $u \in \Lt^{\l_k}(E,F)$ and $v \in \Lt^{\l_l}(F,G)$ then 
we have $$v\circ u \in \Lt^{\l_{k+l}}(E,G)$$ 
and
\[ |v \circ u| \le |v||u|.\]
\end{proposition}
  
In particular for 
$$u_1,u_2,\ldots,u_n \in \Lt^{\l_1}(E,E)$$ then
\[ u_1 \circ u_2\circ \ldots \circ u_n \in \Lt^{\l_n}(E,E)\] 
and we have the estimate:
\[ |u_1 \circ u_2\circ \ldots \circ u_n| \le |u_1| |u_2|\ldots |u_n|, \]
Especially for $u_1=u_2=\ldots=u_n=u$:
\[ |u^n| \le |u|^n .\]
So the graded algebra
\[ \Lt^{\l}(E,F):=\oplus_{k=0}^{\infty}\Lt^{\l_k}(E,F)\]
behaves very much like a Banach algebra and this is the reason that
one is able to develop a functional calculus in this setting.

In reading these formulas one should be aware that for reasons of simplicity
we always write $|-|$ for the norms, but which norm this in fact is, 
depends on the element to which we apply it. For example, 
if $u$ is $\l_1$-local, the  norm in the expression $|u|^n$ refers to 
the $n$-th power of the norm on $\Lt^{\l_1}(E,E)$, whereas in $|u^n|$ we are using the norm on  $\Lt^{\l_n}(E,E)$~! 

\subsection{Submultiplicative functions}

\begin{definition}
A weight function $\l(t,s)$ is {\em submultiplicative} if the sequence
\[ (\l_1,\l_2,\l_3,\ldots) \]
is a submultiplicative sequence, where
\[ \l_n(t,s) := \frac{e^ n \l^n(t,s)}{ n^n}.\]
\end{definition}
The factor $n^n$ appears naturally when we compose local operators while
the factor $e=2.718281828\ldots$ is arbitrary, but leads to a simplification of 
certain statements in functorial calculus due to the fact that~(see for instance~\ref{SS::exponential})
$$\frac{n^n}{e^n} \leq n! $$
As has been remarked before, the function $\l(t,s)=(t-s)$ is submultiplicative.
Note that if  $\l(t,s)$ is submultiplicative, then  
\[ \mu(t,s):=\alpha(s) \beta(t) \l(t,s)\]
is also submultiplicative if  $\a(s)$ and
$\b(t)$ are monotonously increasing and decreasing respectively.
In particular, for $p,q,C \ge 0$, the weight function
\[\D \subset \RM^2_{>0} \to \RM,\  (t,s) \mapsto  C\,s^{p} t^{-q} (t-s) \] 
is submultiplicative. Note that for $pq \neq 0$, the function is however no longer monotonous.
%M signs modified 
From now on we fix  a submultiplicative function $\l(t,s)$ and write
  \[\Lt^k(E,E):=\Lt^{\l_k}(E,E)\]
This space will be called {\em the space of $k$-local operators} (with respect to $\l$).
%%%%%%%%%%%%%%%%%%%%%%%%%%%%%%%%%%%%%%%%%%%%%%%%%%%%%%%%%%%%%%%%%%%%%%

%%%%%%%%%%%%%%%%%%%%%%%%%%%%%%%%%%%%%%%%%%%%%%%%%%%%%%%%%%%%%%%%%%%%%%%%%%%%%%
\subsection{Existence of functorial calculus}
\label{SS::exponential}
From now on we fix a submultiplicative weight $\l$.
Given a power series
$$ f=\sum_{n \geq 0} a_n z^n$$
we use the notation
$$ |f|=\sum_{n \geq 0} |a_n| z^n.$$

\begin{definition}
The {\em Borel transform}  of a formal power series is defined by
$$B:\CM[[z]] \to \CM[[z]],\ \sum_{n \geq 0} a_n z^n \mapsto \sum_{n \geq 0} \frac{a_n}{n!} z^n. $$
\end{definition}
We fix a Kolmogorov space $E$ over $B$ and a submultiplicative weight function 
$\lambda:B \to \RM$.
This defines a Kolmogorov-space $$ \Lt^1(E,E) \to \D^{op}$$ of 1-local operators.
We use the following notation:
\[ \Xt(R):=\{(t,s,u) \in \Lt^1(E,E)\;|\;\; |u| < R \l(t,s) \}\]

\begin{theorem}\label{T::Borel} Let
$f=\sum_{n \geq 0} a_n z^n \in \CM\{z\}$ be a convergent power series with $R$ as radius 
of convergence. Then there is a well-defined map of Banach spaces over $\D \subset \RM^2$, that we call the {\em Borel map}:
$$\Bt f:\Xt(R) \to \Hom(E,E),\ (t,s,v) \mapsto (t,s,\sum_{n=0} \frac{a_n}{n!} v^n) $$
and one has the estimate
$$|\Bt f(u)| \leq |f|\left(\frac{|u|}{\l(t,s)}\right) .$$
\end{theorem}
\begin{proof}
By Proposition \ref{P::algebra}, for $u \in \Lt^1(E,E)$, we have $u^n \in \Lt^n(E,E)$ and 
\[ |u^n| \leq |u|^n \]
Thus there is an inclusion 
$$\p_n: \Lt^n(E,E) \to \Hom(E,E) $$
of Banach-spaces over $\Delta$ and by definition of the rescaled norms
$$ \| \p_n(u) \|  \leq \frac{n^n}{e^n\l(t,s)^n}|u|$$
on the norm. Here $\| \cdot \|$ stands for the standard operator norm inside the Banach space $\Hom(E_t,E_s)$ of Banach operators and $u$ is in fact a triple $(t,s,v)$.

The choice of the constant $e$ gives a simple estimate 
for this norm:
$$\frac{n^n}{e^n\l(t,s)^n} \leq \frac{n!}{\l(t,s)^n} .$$

We obtain:
\[  \|\sum_{n=0}^{\infty} \frac{a_n}{n!} \p_n(u^n)\| \leq \ \sum_{n=0}^{\infty} |a_n| \left(
\frac{|u|}{\l(t,s)}\right)^n= |f|\left(\frac{|u|}{\l(t,s)}\right) .\]
This proves the theorem.
\end{proof} 

\subsection{The exponential map}\index{Exponential}
An important special case is that of the {\em exponential} where we take
\[f=\frac{1}{1-z},\;\;\;B(f) =e^z.\]
The convergence radius $R$ of the series $f$ is equal to $1$. Consider a local
$u \in \Lt^1(E,E)$. 
We then have
\[ \Xt(R)=\{(t,s,u)\in \Lt^1(E,E)\;\;|\;\;|u| < (t-s)\}\]
and get the estimate:
\[ |e^u| \le \frac{1}{1-\frac{|u|}{\l(t,s)}}\]

For instance if we choose for weight  function $\lambda(t,s)=t-s$ and  $u \in \Lt^1(E,E)$ with $|u|\le C$, $\le C.t$, $\le C.t^2$, then
the exponential of $u$ defines a global section over domains of the form:
\begin{center}
\begin{tikzpicture}[scale=2.5]
	\draw[->] (0,0) -- (0,1) node[left] {$s$};
	\draw[->] (0,0) -- (1,0) node[below] {$t$};
	\draw[thin, domain= 0.0:1.0] plot (\x, {\x)});
	\draw[thick,domain= 0.25:1.0, pattern= north east lines](0.25,0)
         plot (\x, {\x-0.25})-- (1,0) -- (0.25,0) -- cycle;
\end{tikzpicture}
\begin{tikzpicture}[scale=2.5]
	\draw[->] (0,0) -- (0,1) node[left] {$s$};
	\draw[->] (0,0) -- (1,0) node[below] {$t$};
	\draw[thin, domain= 0.0:1.0] plot (\x, {\x)});
	\draw[thick,domain= 0.0:1.0, pattern= north east lines](0,0)
         plot (\x, {\x-0.25*\x})-- (1,0) -- (0,0) -- cycle;
\end{tikzpicture}
\begin{tikzpicture}[scale=2.5]
	\draw[->] (0,0) -- (0,1) node[left] {$s$};
	\draw[->] (0,0) -- (1,0) node[below] {$t$};
	\draw[thin, domain= 0.0:1.0] plot (\x, {\x)});
	\draw[thick,domain= 0.0:1.0, pattern= north east lines](0,0)
         plot (\x, {\x-0.25*\x^2})-- (1,0) -- (0,0) -- cycle;
\end{tikzpicture}
\end{center}
\subsection{The derivative (Final part)}
We will show how the formalism works  by itself on domains in a trivial example where we can also make easy explicit computations.
 Consider once more the Kolmogorov space
 $$\Ot^b(D) \to \RM_{>0} , $$
 where $D \to \RM_{>0}$ has fibres $D_t=\{ z \in \CM: |z| <t \} .$
Let us consider the exponential of the derivative:
 $$u(t,s):\Ot^b(D_t) \to \Ot^b(D_s),\ g(z) \mapsto g'(z) $$
According to Taylor's formula
$$e^ug(z)=g(z+1) .$$
As an unbounded operator, the definition domain of the exponential consists of analytic series having a convergence radius $>s+1$.
  What does functorial calculus tells us in such a simple example?
 
 So we consider the weight 
 $$\l:\RM^2_{>0} \to \RM,\ (t,s) \mapsto t-s$$ 
 and consider the derivative as an element $u \in \Lt^1(E,E)$.
Our theorem says that  taking the exponential of the derivation, which corresponds to composition the flow $\p$: 
 $$(e^u)_{t,s}:\Ot^b(D_t) \to \Ot^b(D_s)$$ gives a well defined homomorphism provided we have the estimate:
 $$|u(t,s)|<t-s .$$
 But again by Cauchy inequality we get that $|u(t,s)|=1$ and then we recover the condition $t>s+1$.
 
 So functorial calculus informs us on the way the disk $D_s$ is translated under the flow of the vector field: if the estimates $t>s+1$ holds then the image of the disk $D_s$ under the flow of $\d_z$  at time $1$ is contained inside the disk $D_t$.
 
   In more complicated
  situations domains might be shrinked and moved in a complicated way, we do not know exactly how: our formalism takes care by itself of the necessary information to have a well-defined exponential. 
%%%%%%%%%%%%%%%%%%%%
\subsection{Infinite products}
As we pointed out in the introduction, when dealing with infinite dimensional group actions, it is customary to use iterations where infinite number of transformations are involved~(see e.g. \cite{Moser_Pisa_2}). Our formalism allows us to spell out a theorem which ensure convergence of such iterations. To generalise Theorem~\ref{T::Borel} to infinite products, we need the following elementary lemma, then functorial calculus will do  the rest:
\begin{lemma} Let  $f$ a positive $C^1$-function such that $f(0)=1$.
Then for any real positive summable sequence $(x_i)$, the infinite product $\prod_i f(x_i)$ is convergent.
\end{lemma}
\begin{proof}
Indeed
$$\log \prod_{i \geq 0} f(x_i)=\sum_{i \geq 0} \log(1+x_i+o(x_i)) $$
and $\log(1+y) \sim y$. 
\end{proof}
From this lemma we deduce easily:
\begin{theorem}Let
$f=\sum_{n \geq 0} a_n z^n \in \CM\{z\}$ be a convergent power series with $R$ as radius 
of convergence and $f(0)=1$. Let $(u_i) \subset \Xt(R)$ be such that $\sum_i |u_i|$ is convergent. Then the product
$$g_k:=\Bt f(u_k)\Bt f(u_{k-1})\cdots\Bt f(u_0)=\prod_{i=0}^k \Bt f(u_i) $$ converges
to a horizontal section $g \in \G(A,\Hom(E,E))$
where $A=\{ (t,s): R\l(t,s)>\max |u_i| \}$.
Moreover one has the estimate
$$|\prod \Bt f(u_i)| \leq \prod_{i \geq 0}|f|\left(\frac{|u_i|}{\l(t,s)}\right) .$$
\end{theorem}
\begin{proof}

Consider again the inclusion:
$$\p_n: \Lt^n(E,E) \to Hom(E,E) $$

We use multi-index notations:
$$u^\a=\prod_i u_i^{\a_i},\ a_\a=\prod_i a_{\a_i},\ |\a|=\sum_i \a_i \text{ etc. } $$
where the sequence $\a=(\a_0,\dots, \a_n,\dots)$ has only a finite number of non-zero elements.

We have the formal equality :
$$\sum_\a \frac{a_\a}{|\a|!} x^\a=\prod_{i > 0} Bf(x_i) $$
to which we would like to give a meaning for 
$$x_i=\frac{|u_i|}{\l(t,s)} <R.$$
The homomorphisms $f(x_i)$ exist by Theorem \ref{T::Borel}.  
The estimate:
$$ \| \p_{|\a|}(u^\a) \|  \leq \frac{|\a|!}{\l(t,s)^{|\a|}}|u|^{\a}$$
shows that the series 
$$S= \sum_{\a} a_\a\frac{\p_{|\a|}(u^\a)}{|\a|!}$$
is convergent. Indeed:
$$ \|\sum_{|\a|>k} a_\a\frac{\p_{|\a|} (u^\a)}{|\a|!}\| \leq \sum_{|\a|>k} |a_\a|\frac{|u|^\a}{\l(t,s)^\a} $$
is the rest of a convergent power series
Let us show that the products 
$$g_k:=\Bt f(u_k)\Bt f(u_{k-1})\cdots\Bt f(u_0)$$ 
converge to $S$. We denote by 
$$deg(\a)=\max \{ k \in \NM:\a_k \neq 0 \}$$ 
the last term before the sequence becomes zero. We then have
\begin{align*}
 S(t,s)-g_k(t,s)& = \sum_{deg(\a)>k} a_\a\frac{u^\a}{|\a|!}\\
\| S(t,s)-g_k(t,s) \|& \leq \sum_{deg(\a)>k} |a_\a|\frac{|u|^\a}{\l(t,s)^\a}=\prod_{i > k}|f|\left(\frac{|u_i|}{\l(t,s)}\right)\\
\end{align*}
which is rest of the the convergent infinite product. This proves the theorem.
\end{proof} 
 
\begin{corollary} 
\label{C::compexponential1}
Let $E \to \RM_{>0}$  be a Kolmogorov space, and 
\[(u_i) \subset  L^1(E,E)\] 
a sequence of local morphisms with norm bounded by~$1$. Assume that 
$\s:=\sum| u_i | <+\infty $,
then the sequence
$$g_n:=e^{u_n}e^{u_{n-1}}\cdots e^{u_0}  $$ converges to a horizontal section $g \in \G^h(A,\Hom(E,E))$ where
\[A:=\{(t,s) \in ]0,\tau]^2\;|\; \l(t,s) >\max |u_i| \}.\]
 \end{corollary}
 This gives a practical criterion from which one can deduces the convergence of normal forms theorems, we refer to~\cite{KAM_theory_III} for simple examples of application.

\bibliographystyle{plain}
\bibliography{master.bib}
\end{document}